\documentclass[11pt]{article}
\usepackage{amsthm,amstext}
\usepackage{amsmath}
\usepackage{graphicx}
\usepackage{amsfonts}
\usepackage{amssymb}
\usepackage{latexsym, xypic}
\usepackage{amssymb}
\usepackage[all]{xy}
\xyoption{all} \theoremstyle{plain}\begin{scriptsize}
{\footnotesize \begin{small}
\end{small}}
\end{scriptsize}
\newtheorem{theorem}{Theorem}[section]
\newtheorem{lemma}[theorem]{Lemma}
\newtheorem{proposition}[theorem]{Proposition}

\theoremstyle{definition}
\newtheorem{definition}[theorem]{Definition}
\newtheorem{corollary}[theorem]{Corollary}
\newtheorem{example}{\sc Example}
\theoremstyle{remark}

\date{}

\title{\bf Study Entropy of L-algebras }
\author{\textbf{Mohamad Ebrahimi} \thanks{Corresponding author} ,  \textbf{Arsham Borumand Saeid}\\
 Department of Pure Mathematics, Faculty
of Mathematics and Computer, \\
Shahid Bahonar University of
Kerman, 7616914111, 
Kerman, Iran \\
mohamad\_ebrahimi@uk.ac.ir\\
arsham@uk.ac.ir \\
}
\date{}
\begin{document}

\maketitle

\begin{abstract} In this paper, the concept of L-algebra is revisited and after that, the article is prepared to deal with the notion of the entropy of an L-algebra. If a set has an L-algebraic structure, it is possible to calculate the degree of uncertainty for an event called partition. Sometimes one calculate the degree of uncertainty of an event given another event, which is actually called the conditional entropy. Some propositions are given to calculate the entropy of a dynamical system.
{\bf MSC Classification (2020):} 03E72, 03E74, 94D05, 06A15.
\end{abstract}
\section{Introduction}\label{intr}
 L-algebras are a relatively abstract concept, but they find applications in various fields, including: Logic, Computer Science and algebraic structures.
 They can be used to model logical systems, particularly those involving implication.
They can be applied to formal verification and reasoning about programs.
 They are a specific type of algebraic structure, and their properties can be studied in abstract algebra. L-algebras are a mathematical structure that generalizes the concept of Boolean algebra. They are particularly useful in soft computing, a field that deals with uncertainty, imprecision, and partial truth. One may know that L-algebras have many applications in other Mathematical topics such as group theory, number theory, information theory and so on. On the other hand L-algebras are closely related to some branches of soft computing areas as l-groups, projection lattices of Von Neumann algebras, MV-algebras, braiding and non-commutative logic. Therefore, the importance of L-algebras is not hidden from anyone.
 L-algebra, was introduced by W. Rump in 2008, [18]. This algebraic structure is related to non-classic logical algebras and quantum Yang-Baxter equation solutions. W. Rump proved that every set A
with a binary operation $\rightarrow$ satisfying equation $(x\rightarrow y)\rightarrow (x\rightarrow z) = (y\rightarrow x)\rightarrow (y\rightarrow z)$ corresponds to a solution of the
quantum Yang–Baxter equation if the left multiplication is bijective [21,25]. L-algebras are related to l-groups, projection lattices of Von Neumann algebras, MV-algebras, braiding and non-commutative logic [1,19,22]. In 2012, Yang et al. proved that pseudo MV-algebras can regarded as semi-regular L-algebras [31]. In 2018, W. Rump showed that L-algebras have applications in some theories such as number theory and information theory [24]. In 2019, Wu et al. investigated that each lattice ordered effect algebra creates an L-algebra [29]. In 2022, W. Rump verified that Heyting algebras, MV-algebras, and orthomodular lattices are associated to bounded L-algebras [20,22,23].  \\
Since every L-algebra along with a mapping $l:L\rightarrow L$ is considered a dynamical system, talking about the uncertainty for the occurrence of an event, called partition, is not without grace. Entropy is a useful tool in machine learning to know various concepts such as feature selection, building decision trees, and fitting classification models, etc. Entropy can be considered as an assessment tool of supply chain information [2]. The term
entropy was first used in 1865 by the German physicist, Rudolf Clausius, to denote the thermodynamic function that he introduced in 1854.\\
In 1948, Glude Shanon, the American mathematician, presented the notion of entropy in
information theory. In 1958, the Russian mathematician, Kolmogorov, introduced the concept of measure - theoretic entropy in ergodic theory. Gradually, the notion of entropy was used for the algebraic structures. In 1975, Weiss considered the definition of
entropy for endomorphisms of abelian groups [6].\\
In 1979, Peters gave a different definition of entropy for automorphisms of a discrete abelian group ‘G’. After proving the basic properties, similar to those proved by Weiss, he generalized Weiss’s main result to countable abelian groups, relating the entropy of an automorphism of ‘G’ to the measure theoretic Kolmogorov–Sinai entropy of the adjoint
automorphism of the dual of ‘G’.\\
In 2010, Dikranjan and Giordano Bruno [7, 8] extended the definition of the algebraic entropy, ‘h’, given by Peters for automorphisms to endomorphism of arbitrary abelian groups, and considered its attributes. Specifically, they proved ‘Bridge Theorem’ that connected the algebraic and the topological entropies.\\
Di Nola et al. discussed the entropy of dynamical systems on effect algebras [9], and Rie$\check{c}$an [17]  introduced the concept of the entropy for an MV-algebra dynamical system, and investigated its main features. Ebrahimi represented a different approach to the measure-theoretic notion of entropy using countable partitions, and studied its fundamental properties [10].
In 2010, M.Ebrahimi and N.Mohammadi [11] investigated the entropy of countable partitions of F-structures using the generators of an m-preserving transformation of a discrete dynamical system. In 2015, Mehrpooya, Ebrahimi and Davvaz [16] introduced the entropy of semi independent hyper MV-algebra dynamical systems and investigated its fundamental properties. \\
Ebrahimi and Izadara [13] presented the notion of the entropy of dynamical systems on BCI-algebras. They introduced the concept of ideal entropy of BCI-algebras and investigated its
application in the bilinear codes. M. Maleki, et.al. introduced the notion of entropy on a transitive BE-algebra and investigated the basic properties of this algebraic structure [15]. For more information one can see [31,32,33]. 
\section{Preliminaries}\label{pre}
In this section one can become familiar with the notion of
L-algebras and some their basic properties. \medskip \\
\indent An algebra $(L,\rightarrow,1)$ of type $(2,0)$ is said to be an
L-algebra \cite{18}, if it satisfies  the following conditions:
\begin{itemize}
\item[(1)] $x \rightarrow x = 1$, \item[(2)] $x \rightarrow 1 = 1$, \item[(3)] $1
\rightarrow x = x$, \item[(4)] $(x \rightarrow y) \rightarrow (x \rightarrow z) = (y \rightarrow x) \rightarrow (y \rightarrow z)$, \item[(5)] if  $x \rightarrow y = y \rightarrow x = 1$, then $x=y$, for all $x, y, z \in L$.
\end{itemize}
 The conditions (1), (2) and (3) satisfy that, 1 is a logical unit when the condition (4) is related to the quantum Yang-Baxter equation \cite{21}. 

\begin{example} Let $L = \{a, b, c, 1\}$ and consider the following table. One can investigate that $(L,\rightarrow,1)$ is an L-algebra.
\begin{table} [!h]
$$
\begin{array}{c|cccc}
\rightarrow & a & b & c & 1 \\
\hline
a & 1 & 1 & a & 1 \\
b & a & 1 & c & 1 \\
c & a & 1 & 1 & 1 \\
1 & a & b & c & 1 \\ 
\end{array}
$$
\caption{Example 1} \label{ebr.tex}
\end{table}
\end{example}

 \begin{proposition}{\rm \cite{18}} A logical unit is unique and it is an element of the L-algebra $L$.
\end{proposition}
\indent If $(L, \rightarrow, 1)$ is an L-algebra, then there is a partially order on L, defined by:
\begin{center} 
 $x \leq y$ if and only if $x \rightarrow y = 1,$ for all $x, y\in L $ .
\end{center}
If L has a smallest element 0, then $L$ is called a bounded L-algebra.

\begin{corollary} \cite{18} Let $L$ be an L-algebra. Then the followings are satisfied: 
\begin{itemize}
\item[(1)] $ x\rightarrow (y \rightarrow x) = y \rightarrow (x \rightarrow y)$, 
\item[(2)] $(x \rightarrow y) \rightarrow 1 = (x \rightarrow 1) \rightarrow (y \rightarrow 1)$, $1\rightarrow(x \rightarrow y) = (1 \rightarrow x) \rightarrow (1 \rightarrow y)$. 
\end{itemize}
\end{corollary}

\begin{proposition}  {\rm \cite{18}} Let $L$ be an L-algebra. Then $x\leq y$ implies that $z\rightarrow x \leq z\rightarrow y$.
\end{proposition}
\begin{corollary} \cite{18} Let $L$ be an L-algebra. Then the followings are equivalent:
\begin{itemize}
\item[(1)] $x\leq y\rightarrow x,$ \item[(2)] $x\leq z$ implies $z\rightarrow y \leq x\rightarrow y,$ \item[(3)] $((x\rightarrow y)\rightarrow z)\rightarrow z \leq ((x\rightarrow y)\rightarrow z)\rightarrow ((y\rightarrow x)\rightarrow z),$ for all $x, y, z\in L.$
\end{itemize}
\end{corollary}
\begin{definition}{\rm \cite{18}}  Let $L$ be an L-algebra. A nonempty subset $K$ of $L$ is said to be an L-subalgebra if:
\begin{itemize}
\item[(1)] $1\in K$, \item[(2)]$x\rightarrow y \in K$, for all $x,y\in K$.
\end{itemize}
\end{definition}
\begin{definition} {\rm \cite{14}}
Let $L$ and $K$ be two L-algebras. The map $f: L\rightarrow K$ is said to be an L-algebra homomorphism if:
\begin{itemize}
\item[(1)] $f(1) = 1$, \item[(2)] $f(x\rightarrow y) = f(x)\rightarrow f(y)$, for all $x, y\in L $.
\end{itemize}
\end{definition}

\begin{example} Let $L = \{0, c, a, b, 1\}$. Then $(L, \rightarrow, 1)$ is an L-algebra, when $ \rightarrow $ is defined as follow:
\begin{table} [!h]
$$
  \begin{array}{c|ccccc}
    \rightarrow & 0 & c & a & b & 1 \\
   \hline
    0 & 1 & 1 & 1 & 1 & 1 \\
    c & 0 & 1 & 1 & 1 & 1 \\
    a & 0 & b & 1 & b & 1 \\
    b & 0 & a & a & 1 & 1 \\
    1 & 0 & c & a & b & 1 \\
  \end{array}
  $$
  \caption{Example 2}\label{exam.tex}
\end{table}
\end{example}
Then $(L, \rightarrow, 0, 1)$ is a bounded L-algebra. Now define $f: L\rightarrow L$ by: $f(0) = 0$,  $f(1) = 1$,  $f(a) = b$,  $f(b) = a$ and $f(c) = c$. One can see that $f$ is an L-algebra homomorphism.\\

\begin{definition}
Let $L$ be an L-algebra. A map $l:L\rightarrow  L$ is said to be an L-operator if it satisfies the following conditions:
\begin{itemize}
\item[{\rm(1)}] $x\leq l(x)$,  \item[{\rm(2)}] If $x\leq y$, then $l(x)\leq l(y)$,  \item[{\rm(3)}]  $l(l(x)) = l(x)$,  for all $x, y\in L$.
\end{itemize} 
\end{definition}

\begin{definition}
Let $(L, \rightarrow, 1)$ be an L-algebra. The L-operator $l$, on $L$ is said to be a closure L-operator,  if $l(x\rightarrow y)\leq l(x)\rightarrow l(y)$, for all $x, y\in L$.
\end{definition}
\newpage
\begin{example}
Let $L=\{1, a, b, c, d \}$ and consider the following table:
\begin{table} [!h]
$$
\begin{array}{c|ccccc}
\rightarrow & 1 & a & b & c & d\\
\hline
 1 & 1 & a & b & c & d\\
 a & 1 & 1 & b & c & b\\
 b & 1 & a & 1 & b & a\\
 c & 1 & a & 1 & 1 & a\\
 d & 1 & 1 & 1 & b & 1\\
\end{array}
$$
 \caption{Example 3} \label{exam.tex}
\end{table}
\end{example}
Then $(L, \rightarrow, 1)$ is an L-algebra. Define $l:L\rightarrow L$ by $l(a)=l(1)=1$,   $l(b)=l(d)=b$,   $l(c)=c$. Then $l$ is a closure operator on $L$.\\
Note: From here on, $L$ refers to L-algebra, subalgebra refers to L-subalgebra, operator refers to L-operator.
\begin{definition} Let $l:L\rightarrow L$ be a closure operator on $L$. Then an element $x\in L$ is called simple if $l(x)=x$.
\end{definition}
\begin{proposition} Let $l$ be closure operator on  $L$. Then for each $a\in L$, $l(a) = inf\{x\in L: x\hspace{1.5mm}is \hspace{1.5mm}simple\hspace{1.5mm}and\hspace{1.5mm} a\leq x\}$.
\end{proposition}
\begin{proof}
Let $C = \{x\in L: x\hspace{1.5mm} is \hspace{1.5mm}simple\hspace{1.5mm}and\hspace{1.5mm}a\leq x \}$, and $x_{0} = inf\{x: x\in C \}$. One must show that $l(a) = x_{0}$. Suppose that $x\in C$. Then $x$ is simple, i.e. $l(x) = x$ and $a\leq x$.  $a\leq x$,  then $l(a)\leq l(x) = x$. Therefore $l(a)\leq x$, for each $x\in C$. Then $l(a)\leq inf\{x: x\in C \} = x_{0}$.  $l(l(a))=l(a)$ and $a\rightarrow l(a)=1$. Therefore $l(a)\in C$ and $x_{0}\leq l(a)$. Hence $x_{0}=l(a)$. 
\end{proof}
Let $\Omega(L)$ be the set of all closure operators on $L$ and $l, l^{\prime}\in \Omega(L)$ be arbitrary. Define the partially order $\leq$ on $\Omega(L)$, by:\\
\begin{center}
$l\leq l^{\prime}$ if and only $l(x)\leq l^{\prime}(x)$,  for all $x\in L$.
\end{center}
One can see that $(L, \leq)$ is a partially ordered set.
\begin{definition}
Let $\{l_{i}\}_{i\in N}$ be a family of operators on $L$. Define $inf_{i\in N}{l_{i}}$ by:
\begin{center}
$(inf_{i\in N}{l_{i}})(x)= inf_{i\in N}({l_{i}(x)})$, for all $x\in L$.
\end{center}
\end{definition} 
\begin{proposition} $inf_{i\in N}{l_{i}}$ is an element of $\Omega(L)$, which is the g.l.b of $\{l_{i}\}_{i\in N}$ in the partially ordered set $\Omega(L)$. Therefore $(\Omega(L), \leq)$ is a complete lattice.
\end{proposition}
\begin{proof}
Let $l=inf_{i\in N}{l_{i}}$ and $x\in L$. Since $l_{i}$ is a closure operator for each $i\in N $, one has $x\leq l_{i}(x)$. Then $x\leq inf_{i\in N}{l_{i}(x)}= (inf_{i\in N}{l_{i}})(x)=l(x)$. Now, let $x,y\in L$ be such that $x\leq y$. Thus $l_{i}(x)\leq l_{i}(y)$, for each $i\in N$. Then $l(x)\leq l_{i}(x)\leq l_{i}(y)$, for each $i\in N$. Hence $l(x)\leq (inf_{i\in N}{l_{i}(y)})=(inf_{i\in N}{l_{i}})(y)=l(y)$. Then $l(x)\leq l(y)$. For $x\in L$, it was shown that $l(x)\leq l(l(x))$. $l(x)=(inf_{i\in N}{l_{i}})(x)=inf_{i\in N }{l_{i}(x)}$. But ${l_{i}}({l_{i}}(x))={l_{i}}(x)$, implies that $l(x)=inf_{i\in N}{l_{i}}({l_{i}}(x))=(inf_{i\in N}{l_{i}}){l_{i}}(x))=l({l_{i}}(x)) \geq l(inf_{i\in N }{l_{i}(x)})=l(inf_{i\in N }{l_{i}})(x)=l(l(x)$. Therefore $l(x)=l(l(x))$. Then $l$ is an operator in $\Omega(L)$. It is enough to show that $l$ is a closure operator in $\Omega(L)$. \\

$l(x\rightarrow y)=(inf_{i\in N}{l_{i}})(x\rightarrow y)= inf_{i\in N}({l_{i}}(x\rightarrow y))\leq inf_{i\in N}({l_{i}}(x)\rightarrow {l_{i}}(y))=inf_{i\in N}{l_{i}}(x)\rightarrow inf_{i\in N}{l_{i}(y)}= l(x)\rightarrow l(y)$. Then $l(x\rightarrow y)\leq l(x)\rightarrow l(y)$. \\
Let $l^{\prime}\in \Omega (L)$  be such that $l^{\prime}\leq l_{i}$, for each $i\in N $. Then $l^{\prime}(x)\leq l_{i}(x)$, for each $i\in N$, for each $x\in L$.  Hence $l^{\prime}(x)\leq inf_{i\in N}{l_{i}}(x)$, for each $x\in L$. Therefore $l^{\prime}(x)\leq l(x)$, for each $x\in L$. Then $l^{\prime}\leq l$. This means that $l$ is the g.l.b of $\{l_{i}\}_{i\in N}$.\\\\
On the other hand the operator $\omega_{L}: L\rightarrow L$, which is defined by $\omega_{L}(x)=1$, for each $x\in L$, is a closure operator on $L$. Then $\omega_{L}\in \Omega(L)$, is the greatest element. Therefore $(\Omega(L), \leq)$ is a complete lattice.
\end{proof}
\begin{lemma} Let $\{l_{i}\}_{i\in N}\subset \Omega(L)$ and $l_{0}=sup_{i\in N}{l_{i}}$. Then $l_{0}(x)=x$  if and only if $l_{i}(x)=x$, for all $x\in L$.
\end{lemma}
\begin{proof}
Let $l_{0}(x)=x$, for all $x\in L$. Then $(sup_{i\in N}{l_{i}})(x)=x$, for all $x\in L$. Then $sup_{i\in N}{l_{i}(x)}=x$, for all $x\in L$. Hence $l_{i}(x)\leq x$, for all $x\in L$, and for all $i\in N$. Therefore $l_{i}(x)=x$, for all $x\in L$ and for all $i\in N$.\\
Conversely, let $x$  be a fixed point of $l_{i}$,  for all $i\in N$. Define $\phi: L\rightarrow L$, by:
\begin{center}
$$
\phi(y)=
\begin{cases}
x, & \mbox{if} \hspace{2mm}  y\leq x \\
1, & \mbox{} \hspace{2mm}  otherwise.
\end{cases}
$$
\end{center}
One can claim that $\phi \in \Omega(L)$. Let $y, z\in L$ and $y\leq z$. if $z\rightarrow x=1$. Then $\phi(z)=x$ and $\phi(y)=x$. Hence $\phi(y)=\phi(z)$. But if $z\rightarrow x\neq1$, also by definition of $\phi$, $\phi(z)=1$. Then $\phi(y)\leq \phi(z)$. Let $y\in L$, and $y\rightarrow x=1$. Then $\phi(y)=x$ and $\phi(\phi(y))=\phi(x)=x=\phi(y)$. If $y\rightarrow x\neq 1$. Then $\phi(y)=1$. Therefore $\phi(\phi(y))=\phi(1)=1=\phi(y)$. Since $y\leq \phi(y)$, implies that $\phi$ is a closure operator on $L$. It is shown that $l_{i}\leq \phi$, for all $i\in N$. Suppose that $y\in L$ and $y\rightarrow x=1$. Then $l_{i}(y)\leq l_i(x)=x=\phi(y)$, for all $i\in N$. If $y\rightarrow x\neq 1$. Then $l_{i}(y)\leq 1=\phi(y)$, for all $i\in N$. Therefore $l_{i}\leq \phi$, for all $i\in N$. Then $l_{0}\leq \phi$. Particulary $l_{0}(x)\leq \phi(x)=x$. It is clear that $x\leq l_{0}(x)$. Hence $l_{0}(x)=x$.
\end{proof}
\begin{definition}
An element $l\in \Omega(L)$ is said to be  maximal, if there is no any element $l^{\prime}\in \Omega(L)$, such that $l\leq l^{\prime}\leq \omega_{L}$. In the other word if there exists $l^{\prime}\in \Omega(L)$ that $l\leq l^{\prime}\leq \omega_{L}$. Then $l=l^{\prime}$ or $l^{\prime}=\omega_{L}$.
\end{definition}
\begin{proposition} For $a\in L, a\neq 1$, define $l_{a}:L\rightarrow L$ by:
\begin{center}
$$
l_{a}(x)=
\begin{cases}
a, & \mbox{if}\hspace{2mm}  x\leq a\\
1, & \mbox {}\hspace{2mm}   otherwise.
\end{cases}
$$
\end{center}
Then $l_{a}$ is a maximal element in $\Omega(L)$.
\end{proposition}
\begin{proof}
At first, it is proved that $l_{a}\in \Omega(L)$. Let $x\in L$ and $x\rightarrow a =1$. Then $l_{a}(x)=a$. Then $x\leq a=l_{a}(x)$. If $x\rightarrow a\neq 1$. Then $l_{a}(x)=1$ and hence $x\leq l_{a}(x)$. Therefore $x\leq l_{a}(x)$, for each $x\in L$.\\
Let $x\in L$ and $x\rightarrow a=1$. Then $l_{a}(x)=a$ and $l_{a}(l_{a}(x))= l_{a}(a)=a=l_{a}(x)$. If $x\rightarrow a\neq 1$. Then $l_{a}(x)=1$ and $l_{a}(l_{a}(x))=l_{a}(1)=1=l_{a}(x)$. Thus $l_{a}(l_{a}(x))=l_{a}(x)$, for all $x\in L$.\\
Let $x, y\in L$ and $x\leq y$. If $y\rightarrow a=1$. Then $1=y\rightarrow a\leq x\rightarrow a$, i.e. $x\rightarrow a=1$ and $l_{a}(x)=a=l_{a}(y)$. If $y\rightarrow a\neq
1$. Then $l_{a}(y)=1$ and $l_{a}(x)\leq l_{a}(y)$, for all $x, y \in L $. Hence $l_{a}$ is a closure operator in $\Omega(L).$ \\
Assume that $l^{\prime}\in \Omega(L)$ be such that $l_{a}\leq l^{\prime}\leq \omega_{L}$. If $l^{\prime}(a)=1 $, $x\rightarrow a=1$. Then $a=l_{a}(x)\leq l^{\prime}(x)$. $1=l^{\prime}(a)=l^{\prime}(x\rightarrow a)\leq l^{\prime}(x)$. If $x\rightarrow a\neq 1$. Then $l_{a}(x)=1\leq l^{\prime}(x)$ and $l^{\prime}(x)=1$.  Therefore $l^{\prime}(x)=1$, for all $x\in L$. Then $l^{\prime}=\omega_{L}$. Now, suppose that $l^{\prime}(a)\neq 1$. If $x\rightarrow a=1$. Then $x\leq a$ and $l^{\prime}(x)\leq l^{\prime}(a)$.
$a=l_{a}(x)\leq l^{\prime}(x)$ implies that $l^{\prime}(a)\leq l^{\prime}(l^{\prime}(x))=l^{\prime}(x)$. Thus $l^{\prime}(x)=l^{\prime}(a)$. On the other hand $l_{a}(l^{\prime}(a))\leq l^{\prime}(l^{\prime}(a))=l^{\prime}(a)\neq 1$. Therefore $l^{\prime}(a)\rightarrow a=1$ and $l^{\prime}(a)=a$.  If $x\rightarrow a=1$. Then $l^{\prime}(x)=a=l_{a}(x)$.\\
If $x\rightarrow a \neq 1$. Then $l_{a}(x)=1$ and $l^{\prime}(x)=1$ and hence $l^{\prime}=l_{a}$. Therefore $l^{\prime}=\omega_{L}$ or $l^{\prime}=l_{a}$. Then $l_{a}$ is a maximal element in $\Omega(L)$. 
\end{proof}
\begin{corollary} Let $l$ be a maximal element in $\Omega(L)$. Then there exists $a\in L, a\neq 1$, such that $l=l_{a}$.
\end{corollary}
\begin{proof}
Since $a\neq 1$, $l_{a}\neq \omega_{L}$. Then there exists $x\in L$ such that $l(x)\neq 1$. Let $a=l(x)$. We show that $l=l_{a}$. $l$ is a maximal element of $\Omega(L)$. Then $l_{a}\leq l$. Let $y\in L$ and $y\rightarrow a=1$. Then $y\leq a$ and $l_{a}(y)=a$. Therefore $l(y)\leq a=l_{a}(y)\leq l(y)$. i.e. $l(y)=l_{a}(y)$.\\
If $y\rightarrow a\neq 1$. Then $l_{a}(y)=1$ and $l(y)\leq l_{a}(y)$. Since $l$ is maximal, $l=l_{a}(y)$.  
\end{proof} 
\section{State function, partition and  entropy on L-algebras} \label{ent}
\begin{definition} Let $L$ be a bounded  L-algebra, and $x, y\in L$. Then $x, y$ are said to be orthogonal, denoted by $x\perp y$, if $x\leq y^{\prime}$, where $y^{\prime}=y\rightarrow 0$. If $x, y$ are orthogonal, define: $x\oplus y:=y^{\prime}\rightarrow x$.
\end{definition}
\begin{definition} A state on $L$, is a function $m:L\rightarrow [0,1]$, satisfying the following conditions:\\
i) $m(1)=1$,\\
ii) if $x\perp y$, then $m(x\oplus y)=m(x)+m(y)$,\\
iii) $x\leq y$ implies $m(x)\leq m(y)$.\\
\end{definition}
The state is faithful if $m(x)=0$ implies that  $x=0$, for each $x\in L$. 
\begin{example} 	Let $L=\{0, a, b, 1\}$ and consider the following table.
\begin{table} [!h]
$$
\begin{array}{c|cccc}
\rightarrow & 0 & a & b & 1 \\
\hline
 0 & 1 & 1 & 1 & 1 \\
 a & 0 & 1 & 1 & 1 \\
 b & 0 & 1 & 1 & 1 \\
 1 & 0 & a & b & 1  \\
\end{array}
$$
 \caption{Example 4} \label{exam.tex}
\end{table}
\end{example}
Let $m(0)=0$, $m(a)= m(b)= m(1)=1$, then $m$ is a state on $L$.
\begin{definition} A finite subset $\xi=\{x_{1}, x_{2},..., x_{n}\}$ of a bounded L-algebra $L$, is said to be orthogonal if, for each $k=1, 2,..., n-1$, $\bigoplus_{i=1}^kx_{i} \perp x_{k+1}$. $\xi=\{x_{1}, x_{2}, ..., x_{n}\}$  is said to be a partition of unity (partition for short), if the followings are satisfied:\\
i) $\xi$ is orthogonal,\\
ii) $m(\bigoplus_{i=1}^n x_{i})=1$.
\end{definition}
\begin{definition} Let $\xi=\{x_{1}, x_{2},..., x_{n}\}$ and $\eta=\{y_{1},y_{2},..., y_{m}\}$ be two finite partitions of $L$. $\eta$ is said to be a refinement of $\xi$, denoted by $\xi \eta$, if for each $x_{i}\in \xi$, there exist $y_{i_{1}}, y_{i_{2}},..., y_{i_{k}}\in \eta$ such that $x_{i}=y_{i_{1}} \oplus y_{i_{2}} \oplus ... \oplus y_{i_{k}} $, for ordered subset $\{i_{1}, i_{2}, ..., i_{k}\}$ of $\{1, 2, ..., m\}$.
\end{definition}
\begin{definition} The common refinement of two finite partitions $\xi=\{x_{1}, x_{2},..., x_{n}\}$ and $\eta=\{y_{1}, y_{2},..., y_{m}\}$, is a set denoted by $\xi\vee \eta$, and is defined as:
\begin{center}
$\xi \vee \eta=\{x_{i}\odot y_{j}: 1\leq i\leq n, 1\leq j\leq m\}$.
\end{center}
For each $x, y\in L$, 
$$
x\odot y= 
\begin{cases}
x-y^{\prime}, & \mbox{if} \hspace{2mm} y^{\prime}\leq x\\
0, & \mbox{}\hspace{2mm} otherwise.
\end{cases}
$$
Where $x-y^{\prime}=z$ if and only $x=z \oplus y^{\prime}$.
\end{definition}
\begin{lemma}  {\rm \cite{12}} Let $\xi$ and $\eta$ be two finite partitions of $L$. Then the common refinement $\xi \vee \eta$ is a finite partition of $L$, and $\xi\xi \vee \eta$, $\eta\xi \vee \eta$. 
\end{lemma}
\begin{definition} Let $\xi=\{x_{1}, x_{2}, ..., x_{n}\}$ be a finite partition and $y\in L$. The state $m:L\rightarrow [0,1]$ has the Bay's property if, $m((\bigoplus_{i=1}^n x_{i})\odot y)= m(y)$.
\end{definition}
\begin{lemma} {\rm \cite{12}} If $\xi=\{x_{1}, x_{2}, ..., x_{n}\}$ is a finite partition of a bounded L-algebra $L$, and $m:L\rightarrow [0,1]$ has Bay's property. Then $m((\bigoplus_{i=1}^n x_{i})\odot y)=\sum_{i=1}^n m(x_{i}\odot y)=m(y)$. 
\end{lemma}
Entropy measures the uncertainty of an event. In fact, entropy is not a part of the system, but a means to measure the accuracy of the obtained information. However, entropy plays a significant role in deciding whether the system will continue to work or stop working.
\begin{definition} Let $\xi=\{x_{1}, x_{2}, ..., x_{n}\}$ be a finite partition of a bounded  L-algebra $L$, and  $m:L\rightarrow [0,1]$ be a state on $L$. The entropy of $\xi$ is  
denoted by $H(\xi)$ and is defined as follow:
\begin{center}
$H(\xi)=-\sum_{i=1}^n m(x_{i})\hspace{1mm}log \hspace{1mm}m(x_{i})$.
\end{center}
\end{definition}
\begin{proposition}  {\rm \cite{27}} The function $\phi:[0,\infty)\rightarrow R$ defined by:

$$
\phi(x)=
\begin{cases}
x\hspace{1mm} log\hspace{1mm} x, & \mbox{if} \hspace{2mm} x\neq 0\\
0, & \mbox{if} \hspace{2mm} x=0
\end{cases}
$$
is strictly convex.
\end{proposition}
By the above proposition $H(\xi)$ is well defined and $H(\xi)\geq 0$.
\begin{definition} Let $\xi=\{x_{1}, x_{2}, ..., x_{n}\}$ and $\eta=\{y_{1}, y_{2},..., y_{m}\}$ be two partitions of a bounded L-algebra $L$ and $m:L\rightarrow [0,1]$ be a state on $L$. The entropy of $\xi$ given $\eta$ is denoted by $H(\xi \mid \eta)$, and defined as follow:
\begin{center}
$H(\xi \mid \eta)=-\sum_{i=1}^n \sum_{j=1}^m m(x_{i}\odot y_{j})log\frac{m(x_{i}\odot y_{j})}{m(y_{j})}$.
\end{center}
Omitting the j-terms when $m(y_{j})=0$.
\end{definition}
\begin{definition} If $\xi=\{x_{1}, x_{2}, ..., x_{n}\}$ and $\eta=\{y_{1}, y_{2},..., y_{m}\}$ are two partitions of a bounded L-algebra $L$ and $m:L\rightarrow [0,1]$ is a state on $L$. $\xi$ is said to be an interior subset of $\eta$ and denoted by $\xi \stackrel{\circ}\subseteq \eta$, if the following holds:
\begin{center}
$\xi\stackrel{\circ}\subseteq \eta$ if for every $y_{j}\in \eta$ there exists $x_{i}\in \xi$ that $m(x_{i}\odot y_{j})=m(y_{j})$.
\end{center}
 Note that  $\xi \stackrel{\circ}=\eta$ if and only if  $\xi\stackrel{\circ}\subseteq \eta$ and $\eta\stackrel{\circ}\subseteq \xi$.
\end{definition}
If $\xi=\{1\}$. Then $H(\xi)=0$, and it means that there is no any uncertainty i.e the event $\xi=\{1\}$ will occur. If $\xi=\{x_{1}, x_{2}, ...,x_{n}\}$ and $m(x_{i})=\frac{1}{n}$ for all $1\leq i\leq n$. Then $H(\xi)=log\hspace{1mm}n$.  In this case there is the maximum uncertainty  $log \hspace{1mm}n$.
\newpage
\begin{example} Let $L=\{0, a, b, 1\}$ be as in example 4.
\begin{table} [!h]
$$
\begin{array}{c|cccc}
\rightarrow & 0 & a & b & 1 \\
\hline
 0 & 1 & 1 & 1 & 1 \\
 a & 0 & 1 & 1 & 1\\
 b & 0 & 1 & 1 & 1\\
 1 & 0 & a & b & 1\\
\end{array}
$$
 \caption{Example 5} \label{exam.tex}
\end{table}
\end{example}
Define $m:L\rightarrow [0,1]$ by $m(0)=0$, $m(1)=m(a)=m(b)=1$. Let $\xi=\{0, a\}$ and $\eta=\{0, b\}$. One can check that  $\xi\stackrel{\circ}= \eta$.
\begin{proposition} Let $\xi$ and $\eta$ be two finite partitions of a bounded L-algebra $L$, and $m:L\rightarrow [0,1]$ be a state having Bay's property. Then 
\begin{center}
$H(\xi \vee \eta)=H(\xi \mid \eta)+H(\eta)$.
\end{center}
\end{proposition}
\begin{proof} Suppose that $\xi=\{x_{1}, x_{2},..., x_{n}\}$ and $\eta=\{y_{1}, y_{2},..., y_{m}\}$. 
\begin{center}
$H(\xi \mid \eta)=-\sum_{i=1}^n \sum_{j=1}^m m(x_{i}\odot y_{j})log\frac{m(x_{i}\odot y_{j})}{m(y_{j})}=-\sum_{i=1}^n\sum_{j=1}^m m(x_{i}\odot y_{j})log(m(x_{i}\odot y_{j}))+\sum_{i=1}^n\sum_{j=1}^m m(x_{i}\odot y_{j})log(m(y_{j}))=H(\xi \vee \eta)+\sum_{j=1}^m\sum_{i=1}^n m(x_{i}\odot y_{j})log(m(y_{j}))=H(\xi \vee \eta)- (-\sum_{j=1}^m m(y_{j})log(m(y_{j}))=H(\xi \vee \eta)-H(\eta)$.
\end{center}
\end{proof}
\begin{definition} {\rm \cite{18}} An L-algebra $L$ is said to be self  distributive if
\begin{center}
$(x\odot (y\odot z))=(x\odot y)\odot(x\odot z)$,\\
or\\
$((x\odot y)\odot z))=(x\odot z)\odot(y\odot z)$,  for all $x, y, z\in L$.
\end{center} 
\end{definition}
\begin{definition} Let $\xi$ and $\eta$ be two finite partitions of a bounded L-algebra $L$ and $m:L\rightarrow [0,1]$ be a state on $L$. $\xi$ and $\eta$ are said to be independent if
\begin{center}
$m(x_{i}\odot y_{j})=m(x_{i})\hspace{1mm}m(y_{j})$, for $i=1, 2, ..., n$ and $j=1, 2, ..., m$.
\end{center}
\end{definition}
\begin{proposition} Let $\xi$ and $\eta$ be two partitions of a bounded L-algebra $L$. Let $m:L\rightarrow [0,1]$ be a state having the Bay's property. Then \\
i) $\xi\stackrel{\circ}\subseteq \eta$ if and only if $H(\xi\mid \eta)=0$,\\
ii) $\xi \stackrel{\circ}=\eta$ implies $H(\xi)=H(\eta)$,\\
iii) If $L$ is self distributive, $\xi \odot \eta$ and $\zeta$, also $\eta \odot \zeta$ and $\xi$ are independent. Then $\xi \stackrel{\circ}=\eta$ implies $H(\xi \mid \zeta)=H(\eta \mid \zeta)$, for each finite partition $\zeta$ of $L$,\\
iv) If $L$ is self distributive and  $\eta \stackrel{\circ}=\zeta$. Then $H(\xi \mid \zeta)= H(\xi \mid \eta)$, for each finite partition $\xi$ of $L$.
\end{proposition}
\begin{proof} Suppose that $\xi\stackrel{\circ}\subseteq \eta$. Then for each $y_{j}\in \eta $ there exists $x_{i_{0}}\in \xi$ such that $m(y_{j})=m(x_{i_{0}}\odot y_{j})$. Since $m$ has Bay's property, $m(y_{j})=\sum_{i=1}^n m(x_{i}\odot y_{j})$. Hence $m(x_{i}\odot y_{j})=0$, for all $i\neq i_{0}$. \\
$H(\xi\mid \eta)=\sum_{i=1}^n\sum_{j=1}^m m(x_{i}\odot y_{j})log\frac{m(y_{j})}{m(x_{i}\odot y_{j})}=m(x_{i_{0}}\odot y_{j})log\frac{m(y_{j})}{m(y_{j})}=0$.\\
Conversely, suppose that $H(\xi\mid \eta)=0$. Then $H(\xi\mid \eta)=\sum_{i=1}^n\sum_{j=1}^m m(x_{i}\odot y_{j})log\frac{m(y_{j})}{m(x_{i}\odot y_{j})}=0$. Since $m(x_{i}\odot y_{j})\geq 0$ and $log\frac{m(y_{j})}{m(x_{i}\odot y_{j})}\geq 0$, for each $1\leq i\leq n$ and $1\leq j\leq m$, $m(x_{i}\odot y_{j})=0$ or $m(x_{i}\odot y_{j})=m(y_{j})$.\\
For arbitrary $y_{j}\in \eta$, $0\neq m(y_{j})=\sum_{i=1}^n m(x_{i}\odot y_{j})$. Then there exists $1\leq i_{0}\leq n$ such that $m(y_{j})=m(x_{i_{0}}\odot y_{j})$. Therefore $\xi\stackrel{\circ}\subseteq \eta$.\\
ii) Let $\xi \stackrel{\circ}=\eta$. Since $\xi \stackrel{\circ} \subset\eta$, $H(\xi\mid \eta)=0$. Then by Proposition 3.13, $H(\xi \vee \eta)=H(\eta)$. Similarly if $\eta \stackrel{\circ}\subset \xi$, then $H(\xi \vee \eta)=H(\xi)$. Therefore $H(\xi)=H(\eta)$.\\
iii) At first assume that $\xi\stackrel{\circ}\subseteq \eta$ and $y_{j}\odot z_{k}\in \eta \vee \zeta$. $y_{j}\in \eta$ and $\xi \stackrel{\circ}\subseteq \eta$. Then there exists $x_{i}\in \xi$ such that $m(x_{i}\odot y_{j})=m(y_{j})$.\\
$m((x_{i} \odot z_{k})\odot (y_{j}\odot z_{k})=m((x_{i} \odot y_{j})\odot z_{k})=m(x_{i}\odot y_{j})m(z_{k})=m(y_{j})m(z_{k})=m(y_{j}\odot z_{k})$, for $x_{i}\odot z_{k}\in \xi \odot \zeta$. Then $\xi \vee \zeta \stackrel{\circ} \subseteq \eta \vee \zeta$. Similarly from $\eta \stackrel{\circ}\subseteq \xi$, implies that $\eta \vee \zeta \stackrel{\circ}\subseteq \xi \vee \zeta$. \\
Therefore $\xi \vee \zeta \stackrel{\circ}= \eta \vee \zeta$. Then $H(\xi \vee \zeta)=H(\eta \vee \zeta)$, hence $H(\xi \mid \zeta)+H(\zeta) =H(\eta \mid \zeta)+H(\zeta)$ and $H(\xi \mid \zeta)=H(\eta \mid \zeta)$.\\
iv) If $\zeta \stackrel{\circ}= \eta$. Then $\xi \vee \zeta \stackrel{\circ}= \xi \vee \eta$, for each finite partition $\xi$ of $L$. Therefore $H(\xi \vee \zeta)=H(\xi \vee \eta)$.  Hence $H(\xi \mid \zeta)+H(\zeta)=H(\xi \mid \eta)+H(\eta)$. Since $\eta \stackrel{\circ}= \zeta$, $H(\eta)=H(\zeta)$, and hence $H(\xi \mid \zeta)=H(\xi \mid \eta)$.
\end{proof}
\begin{corollary} Let $\mathbb{N}=\{1\}$ and $\xi$ be a finite partition of a bounded L-algebra $L$. Let $m:L\rightarrow [0,1]$ be a state on $L$. Then $H(\xi \mid \mathbb{N})=H(\xi)$.
\end{corollary}
\begin{proof} Let $\xi= \{x_{1}, x_{2}, ..., x_{n}\}$. $x_{i}\odot 1=x_{i}-1^{\prime}=x_{i}-0=x_{i}$. Then $H(\xi \mid \mathbb{N})= -\sum_{i=1}^n m(x_{i} \odot 1)log\frac{m(x_{i}\odot 1)}{m(1)}=-\sum_{i=1}^n m(x_{i})log\hspace{1mm}m(x_{i})=H(\xi)$.
\end{proof}
\begin{proposition} Let $\xi=\{x_{1}, x_{2},..., x_{n}\}$, $\eta=\{y_{1}, y_{2},..., y_{m}\}$ and $\zeta=\{z_{1}, z_{2},..., z_{k}\}$ such that $\xi$ and $\zeta$, $\xi$ and $\eta$, $\xi \vee \eta$ and $\zeta$, $\eta$ and $\xi \vee \zeta$ be independent. Let $m:L\rightarrow [0,1]$ be a state on $L$, having Bay's property. Then $H(\xi \vee \eta \mid \zeta)=H(\xi \mid \zeta)+ H(\eta \mid \xi \vee \zeta)$.
\end{proposition}
\begin{proof} Assume that $\xi= \{x_{1}, x_{2}, ..., x_{n}\}$, $\eta=\{y_{1}, y_{2},..., y_{m}\}$ and $\zeta=\{z_{1}, z_{2},..., z_{k}\}$.
\begin{center}
$H(\xi \vee \eta \mid \zeta)=-\sum_{i=1}^n \sum_{j=1}^m \sum_{l=1}^k m((x_{i} \odot y_{j})\odot  z_{l})log \frac{m((x_{i} \odot y_{j})\odot z_{l})}{m(z_{l})}=-\sum_{i=1}^n \sum_{j=1}^m \sum_{l=1}^k m((x_{i} \odot y_{j})\odot z_{l})log \frac{m((x_{i} \odot y_{j})\odot z_{l})}{m(z_{l})}\frac{m(x_{i}\odot z_{l})}{m(x_{i}\odot z_{l})}=-\sum_{i=1}^n \sum_{j=1}^m \sum_{l=1}^k m((x_{i} \odot y_{j})\odot z_{l})log \frac{m((x_{i} \odot y_{j})\odot z_{l})}{m(x_{i}\odot z_{l})} -\sum_{i=1}^n \sum_{j=1}^m \sum_{l=1}^k m((x_{i} \odot y_{j})\odot z_{l})log \frac{m((x_{i} \odot z_{l})}{m(z_{l})}=-\sum_{i=1}^n \sum_{j=1}^m \sum_{l=1}^k m((y_{j} \odot (x_{i}\odot z_{l}))log \frac{m((y_{j} \odot (x_{i}\odot z_{l}))}{m(x_{i}\odot z_{l})}-\sum_{i=1}^n \sum_{l=1}^k m(x_{i} \odot z_{l})log \frac{m((x_{i}\odot z_{l})}{m(z_{l})}=H(\eta \mid \xi \vee\zeta)+H(\xi \mid \zeta)$.
\end{center}
\end{proof}
\begin{corollary} Let $\xi_{1}, \xi_{2},...,\xi_{n}, \eta$ be partitions of a bounded L-algebra $L$ and $m:L\rightarrow [0,1]$ be a state on $L$, having Bay's property. Let $\xi_{0}=\mathbb{N}=\{1\}$. Then the followings hold:\\
i) $H(\xi_{1}\vee\xi_{2}\vee...\vee\xi_{n})=\sum_{i=1}^nH(\xi\mid \bigvee_{k=0}^{i-1} \xi_{k})$, \\
ii) $H(\bigvee_{i=1}^n \xi_{i} \mid\eta)=\sum_{i=1}^n H(\xi_{i}\mid (\bigvee_{k=0}^{i-1}\xi_{k})\vee \eta)$.
\end{corollary}
\begin{lemma} Let $\xi$ and  $\eta$ be two finite partitions of a bounded L-algebra $L$. Let $m:L\rightarrow [0,1]$ be a state on $L$. If $\xi\eta$ and $Card(\eta) \leq Card(\xi)$. Then $H(\xi)\leq H(\eta)$.
\end{lemma}
\begin{proof} Let $\xi=\{x_{1}, x_{2},..., x_{n}\}$ and $\eta=\{y_{1}, y_{2},..., y_{m}\}$.  Furthermore assume that $m\leq n$. Since $\xi \eta$, for each $x_{i}\in \xi$ there exist $y_{i_{1}}, y_{i_{2}},..., y_{i_{k}}\in \eta$ such that $x_{i}=y_{i_{1}}\oplus y_{i_{2}}\oplus ...\oplus y_{i_{k}}$.\\
$H(\xi)=-\sum_{i=1}^n m(x_{i})log m(x_{i})=-\sum_{i=1}^n\sum_{l=1}^{k_{i}}m(y_{i_{l}})log(\sum_{l=1}^{k_{i}}m(y_{i_{l}})\leq -\sum_{j=1}^m m(y_{j})log\hspace{1mm}m(y_{j})=H(\eta)$.
\end{proof}
\begin{proposition} Let $\xi$ and $\eta$ be two finite partitions of a bounded L-algebra $L$.  Furthermore suppose that $L$ is commutative and self distributive with respect to $\odot$. Then $H(\xi \mid \eta)\leq H(\xi)$. 
\end{proposition}
\begin{proof} Let $\xi=\{x_{1}, x_{2}, ..., x_{n}\}$ and $\eta=\{y_{1}, y_{2}, ..., y_{m}\}$. Since $L$ is self distributive, $x_{i}\odot y_{j}= y_{j}^{\prime}\rightarrow x_{i}$ and $x_{i}\leq y_{j}^{\prime}\rightarrow x_{i}$ and hence $x_{i}\leq x_{i}\odot y_{j}$, for all $1\leq i\leq n$ and $1\leq j\leq m$. Then $m(x_{i})\leq m(x_{i}\odot y_{j})\leq \frac{m(x_{i}\odot y_{j})}{m(y_{j})}$. Since $L$ is commutative, $m(x_{i})log\hspace{1mm}m(x_{i})\leq m(x_{i}\odot y_{j})log\frac{m(x_{i}\odot y_{j})}{m(y_{j})}$, for all $1\leq i\leq n$ and $1\leq j\leq m$. Hence $\sum_{i=1}^n m(x_{i})log\hspace{1mm}m(x_{i})\leq \sum_{i=1}^n m(x_{i}\odot y_{j})log\frac{m(x_{i}\odot y_{j})}{m(y_{j})}$, for all $1\leq j\leq m$. Then $\sum_{i=1}^n m(x_{i})log\hspace{1mm}m(x_{i})\leq \sum_{i=1}^n \sum_{j=1}^m m(x_{i}\odot y_{j})log\frac{m(x_{i}\odot y_{j})}{m(y_{j})}$. Therefore $H(\xi \mid \eta)\leq H(\xi)$.
\end{proof}
\begin{proposition} Let $\xi\eta$ and $\zeta$ be a finite arbitrary partition of a bounded L-algebra $L$, $\zeta \neq \mathbb{N}$ and $L$ be self distributive. Then $H(\xi \mid \zeta)\leq H(\eta \mid \zeta)$
\end{proposition}
\begin{proof} Let $\xi=\{x_{1}, x_{2}, ..., x_{n}\}$ and $\eta=\{y_{1}, y_{2}, ..., y_{m}\}$ and $\zeta=\{z_{1}, z_{2}, ..., z_{s}\}$. Since $\xi\eta$, for each $1\leq i\leq n$ there are $y_{i_{1}}, y_{i_{2}}, ..., y_{k_{i}}\in \eta$ such that $x_{i}=\sum_{l=1}^{k_{i}}y_{i_{l}}$. \\
 Then $\xi \vee \zeta=\{x_{i}\odot z_{t}: 1\leq i\leq n, 1\leq t\leq s \}$ and $\eta \vee \zeta=\{y_{j}\odot z_{t}: 1\leq j\leq m, 1\leq t \leq s\}$.\\
$x_{i}\odot z_{t}=\sum_{l=1}^{k_{i}} {y_{j_{l}}}\odot z_{t}$, for all $1\leq i\leq n, 1\leq t\leq s$. Then $\xi \vee \zeta  \eta \vee \zeta$. By Proposition 3.13, $H(\xi \vee \zeta)=H(\xi \mid \zeta)+H(\zeta)$ and $H(\eta \vee \zeta)=H(\eta \mid \zeta)+H(\zeta)$. But $H(\xi \vee \zeta)\leq H(\eta \vee \zeta)$ and then $H(\xi \mid \zeta)+H(\zeta)\leq H(\eta \mid \zeta)+	H(\zeta)$. Therefore $H(\xi \mid \zeta)\leq H(\eta\mid \zeta)$.
\end{proof}
\begin{corollary} Let $\xi$ and $\eta$ be two finite partitions of a bounded L-algebra $L$, and $m:L\rightarrow [0,1]$ be a state on $L$, having Bay's property. Then\\
i) $H(\xi \vee \eta)\leq H(\xi)+H(\eta)$,\\
ii) $H(\xi \mid \eta)=H(\xi)$ if and only if $H(\xi \vee \eta)=H(\xi)+H(\eta)$ if and only if $\xi$ and $\eta$ are independent.
\end{corollary}
\begin{definition} Let $(L,\rightarrow,1)$ be bounded L-algebra, a dynamical system on $L$ is a triple $(L,T,m)$, where $m:L\rightarrow [0,1]$ is a state on $L$ with the Bay's property and $T: L\rightarrow L$ is a mapping that has the following properties:\\
i) $T(a\rightarrow b)= T(a)\rightarrow T(b)$, for each $a,b\in L$,\\
ii) $T(1)=1$,\\
iii) $m(T(a))= m(a)$, for each $a\in L$. 
\end{definition}
\begin{lemma} Let $(L,T,m)$ be dynamical system on $L$ or L-system for short and $\xi$ be a  finite partition of $L$. Then $T^n(\xi)=\{T^n(x): x\in \xi\}$ is a finite  partition of $L$.
\end{lemma}
\begin{proof} Let $\xi=\{x_{1},x_{2}, ..., x_{m}\}$ be a finite partition of $L$. Then $m(\oplus_{i=1}^m x_{i})=1$ and $\oplus_{i=1}^m x_{i} \bot x_{k+1}$, for $k=1, 2, ..., m-1$. $T^n(\xi)=\{T^n(x_{1}), T^n(x_{2}), ..., T^n(x_{m}) \}$, $T(a\rightarrow b)= T(a)\rightarrow T(b)$ and then $T(a\oplus b)=T(b^{\prime}\rightarrow a)=T(b^{\prime})\rightarrow T(a)=T(b\rightarrow 0)\rightarrow T(a)=(T(b)\rightarrow T(0))\rightarrow T(a)=T(b)^{\prime}\rightarrow T(a)=T(a)\oplus T(b)$. \\
One can show that $T(\oplus_{i=1}^m x_{i})=\oplus_{i=1}^m T(x_{i})$. Then $T^s (\oplus_{i=1}^m x_{i})=\oplus_{i=1}^m T^s (x_{i})$, for each $s\in N$. Now $m(\oplus_{i=1}^m T^s(x_{i}))
=m(T^s(\oplus_{i=1}^m x_{i}))=m(\oplus_{i=1}^m x_{i})=1$, for each $s\in N$. $T^s(x_{i})\bot T^s_{x_{j}}$, for each $i\neq j, 1\leq i,j\leq m$. 	Because $T(a^{\prime})=T(a\rightarrow 0)=T(a)\rightarrow T(0)=T(a)^{\prime}$, for each $a\in L$. \\
It is enough to consider $(T^s(x_{i}))^{\prime}=T^s(x_{i})\rightarrow (0=T^s(0))=T^s(x_{i}\rightarrow0)=T^s(x_{i}^{\prime})$, for each $s\in N$. Since $x_{i}, x_{j}\in  \xi$, $x_{i}\bot x_{j}$, then $x_{j} \leq x_{i}^{\prime}$ and $T(x_{j})\leq T(x_{i}^{\prime})=(T(x_{i}))^{\prime}$. Hence $T(x_{j})\perp T(x_{i})$. Therefore $T^n(\xi)$ is a finite partition of $L$.
\end{proof}
\begin{lemma} Let $(L,T,m)$ be an L-system and $\xi, \eta$ be two finite partitions of $L$. Then the followings are satisfied:\\
i) $H(T^n(\xi))=H(\xi)$,\\
ii) $H(T^n(\xi) \mid T^n(\eta))=H(\xi \mid\eta),$\\
iii) $H(\bigvee_{i=1}^{n-1} T^i(\xi)=H(\xi)+\sum_{j=1}^ {n-1}H(\xi\mid \bigvee_{i=1}^jT^i(\xi))$ 
\end{lemma}
\begin{proof} iii) By induction on $n$ it is proved. For $n=1$ it is clear. Assume that for $n=k$  is true. We prove the equality for $n=k+1$.\\
$H(\bigvee_{i=0}^kT^i(\xi))=H(\bigvee_{i=1}^kT^i(\xi)\bigvee \xi)=H(\bigvee_{i=1}^kT^i(\xi))+H(\xi\mid\bigvee_{i=1}^kT^i(\xi))=H(\bigvee_{i=0}^{k-1}T^{i+1}(\xi))+H(\xi\mid\bigvee_{i=1}^kT^i(\xi))=H(\bigvee_{i=0}^{k-1}T^i(\xi))+H(\xi\mid\bigvee_{i=1}^kT^i(\xi))=H(\xi)+\sum_{j=1}^{k-1}H(\xi\mid\bigvee_{i=1}^jT^i(\xi))+H(\xi\mid\bigvee_{i=1}^kT^i(\xi))=H(\xi)+\sum_{j=1}^kH(\xi\mid\bigvee_{i=1}^jT^i(\xi))$
\end{proof}
\begin{proposition} Let $\xi$ be a finite partition of $(L,T,m)$. Then the $lim_{n\rightarrow\infty}H(\bigvee_{i=1}^nT^i(\xi))$ exists.
\end{proposition}
\begin{proof} Let $a_{n}=H(\bigvee_{i=1}^nT^i(\xi))$. One can show that $a_{n+p}\leq a_{n}+a_{p}$, for each $n,p\in N$. Then by   Theorem the proof is completed.
\end{proof}
\begin{definition} Let $(L,T,m)$ be an $L$-system and $\xi$ be a finite partition of $L$. Then the entropy of the dynamical system with respect to $\xi$ is defined as:\\
\begin{center}
$h(T, \xi)=lim_{n\rightarrow\infty}\frac{1}{n}H(\bigvee_{i=0}^{n-1}T^i(\xi))$.
\end{center} 
\end{definition}
By Proposition 3.30, $h(T,\xi)$ is well define.
\begin{proposition} Let $\xi$ be a finite partition of an $L-system$ $(L,T,m)$. Then the followings are satisfied:\\
i) $h(T,\xi)\leq H(\xi)$,\\
ii) $h(T,\xi\vee \eta)\leq h(T,\xi)+h(T,\eta)$,\\
iii) if  $\xi \stackrel{\circ}\subset \eta $, then $h(T,\xi)\leq h(T,\eta)$,\\
iv) if $\odot$ is self distributive, then $h(T,\xi)\leq h(T,\eta)+H(\xi\mid\eta)$,\\
v) $h(T,T(\xi))=h(T,\xi)$,\\
vi) $h(T, \bigvee_{i=0}^{n-1}T^i(\xi))=h(T,\xi)$.
\end{proposition}
\begin{proof} i) As mentioned above $T(\xi)$ is a finite partition of $L$. $h(T,\xi)=lim_{n\rightarrow\infty}\frac{1}{n}H(\bigvee_{i=0}^{n-1}T^i(\xi))\leq lim_{n\rightarrow \infty}\sum_{i=0}^{n-1}H(T^i(\xi))=-lim_{n\rightarrow \infty}\sum_{i=0}^{n-1}\sum_{j=1}^k m(T^i(x_{j}))logm(T^i(x_{j})=lim_{n\rightarrow \infty}\frac{1}{n}\sum_{i=0}^{n-1}\sum_{j=1}^k m(x_{j})logm(x_{j})= lim_{n\rightarrow \infty}\frac{1}{n}\sum_{i=0}^{n-1} H(\xi)=H(\xi)$.\\
ii) At first we show that $T(a\oplus b)=T(a)\oplus T(b)$. One saw that $T(a^{\prime})=(T(a))^{\prime}$. $T(a\oplus b)=T(b^{\prime} \rightarrow a)=(T(b))^{\prime}\rightarrow T(a))= T(a)\oplus T(b)$. Now one conclude that $T(a-b)=T(a)-T(b)$, for all $a,b \in L$. Let $a-b=c$, then $a=c\oplus b$ and $T(a)=T(c\oplus b)= T(c)\oplus T(b)$. Hence $T(c)=T(a-b)=T(a)-T(b)$.  $h(T,\xi \vee \eta)=lim_{n\rightarrow \infty}\frac{1}{n}H(\bigvee_{i=0}^{n-1}T^i(\xi \vee \eta))=lim_{n\rightarrow \infty}\frac{1}{n}H(\bigvee_{i=0}^{n-1}(T^i\xi\vee T^i\eta))=lim_{n\rightarrow \infty}\frac{1}{n}H((\bigvee_{i=0}^{n-1}T^i\xi)\bigvee(\bigvee_{i=0}^{n-1}T^i \eta))\leq lim_{n\rightarrow \infty}\frac{1}{n}H(\bigvee_{i=0}^{n-1}T^i \xi)+lim_{n\rightarrow \infty}\frac{1}{n}H(\bigvee_{i=0}^{n-1}T^i \eta)=h(T, \xi)+h(T,\eta)$.\\
iii) Suppose that  $\xi \stackrel{\circ}\subset \eta $, then $H(\xi \mid \eta)=0$ and $H(\xi \vee \eta)=H(\eta)$. $\xi \vee \eta = \{x_{i} \odot y_{j}: 1\leq i\leq n,1\leq j\leq m\}$. $x_{i}\odot y_{j}=x_{i}- y_{j}^{\prime}$, and then $x_{i}=z\oplus y_{j}^{\prime}$. Hence $m(x_{i})=m(z\oplus y_{j}^{\prime})=m(z)+m(y_{j}^{\prime})$. Therefore $m(x_{i}\odot y_{j})\leq m(x_{i})$, for each $1\leq i\leq n$ and $1\leq j\leq m$. Since $log$ is an increasing function, $m(x_{i}\odot y_{j})log \hspace{1mm}m(x_{i}\odot y_{j})\leq m(x_{i})log \hspace{1mm}m(x_{i})$, for each $1\leq i\leq n, 1\leq j\leq m$. $\sum _{j=1}^m \sum _{i=1}^n m(x_{i}\odot y_{j})log \hspace{1mm}m(x_{i}\odot y_{j})\leq m\sum_{i=1}^n m(x_{i})log\hspace{1mm} m(x_{i})$. Then $H(\xi \vee \eta)\geq mH(\xi)\geq H(\xi)$.\\
$H(\xi \vee \eta)=H(\eta)+H(\xi \mid \eta)$. Then $H(\xi)\leq H(\xi \vee \eta)=H(\eta)$. $H(\bigvee_{i=0}^{n-1} T^i \xi)\leq H(\bigvee_{i=0}^{n-1} T^i \eta)$. Hence $h(T,\xi)\leq h(T,\eta)$.\\
iv) $H(\bigvee_{i=0}^{n-1}T^i \xi)\leq H((\bigvee_{i=0}^{n-1}T^i \xi) \bigvee(\bigvee _{i=0}^{n-1}T^i \eta))=H(\bigvee_{i=0}^{n-1}T^i\eta)+H((\bigvee_{i=0}^{n-1}T^i\xi) \mid (\bigvee_{i=0}^{n-1}T^i\eta))\leq H(\bigvee_{i=0}^{n-1}T^i \eta)+\sum_{i=0}^{n-1}H(T^i \xi \mid \bigvee_{i=0}^{n-1}T^i \eta)\leq H(\bigvee_{i=0}^{n-1}T^i \eta)+\sum_{i=0}^{n-1} H(T^i \xi \mid T^i \eta)=H(\bigvee_{i=0}^{n-1}T^i \eta)+\sum_{i=0}^{n-1}H(\xi \mid \eta)$.\\
v) $H(\bigvee_{i=0}^{n-1}T^i \xi)=H(\bigvee_{i=1}^n T^{i-1}\xi)=H(\bigvee_{i=0}^{n-1}T^i(T^{-1}\xi))$, and then $h(T,\xi)=h(T,T^{-1}\xi)$.\\
vi) $h(T,\bigvee_{i=0}^k T^i \xi)=lim_{n\rightarrow \infty}\frac{1}{n}H(\bigvee_{j=0}^{n-1}T^j(\bigvee_{i=0}^k T^i \xi))=lim_{n\rightarrow \infty}\frac{1}{n}H(\bigvee_{i=0}^{k+n-1}T^i \xi)=lim_{n\rightarrow \infty}\frac{k+n-1}{n}\frac{1}{k+n-1}H(\bigvee_{i=0}^{k+n-1}T^i \xi)=lim_{n\rightarrow \infty}\frac{1}{k+n-1}H(\bigvee_{i=0}^{k+n-1}T^i \xi)=h(T,\xi)$.
\end{proof}
\begin{definition} Let $(L,T,m)$ be an L-system. The entropy of the dynamical system $T:L \rightarrow L$ is defined as:
\begin{center}
$h(T)=sup\{h(T,\xi): \xi\hspace{1mm} is\hspace{1mm} a \hspace{1mm}finite \hspace{1mm}partition\hspace{1mm} of \hspace{1mm}L\}$
\end{center}
\end{definition}
It is clear that $0\leq h(T)\leq \infty$ and $h(id)=0$.\\
\begin{definition} Let $(L_{1}, T_{1},m_{1})$ and $(L_{2},T_{2},m_{2})$ be two L-systems. The dynamical systems $T_{1},T_{2}$ are said to be isomorphic if there exists $\phi: L_{1}\rightarrow L_{2}$ such that \\
i) $\phi(a\rightarrow b)=\phi(a)\rightarrow\phi(b)$, for each $a,b\in L_{1}$,\\
ii) $\phi\hspace{1mm} o \hspace{1mm}T_{1}=T_{2}\hspace{1mm}o \hspace{1mm}\phi$,\\             
iii) $m_{2}\hspace{1mm}o\hspace{1mm}\phi=m_{1}$.
\end{definition}
One can investigate that:\\
i) $\phi(0)=0$ and $\phi(1)=1$,\\
ii) $\phi(a\oplus b)=\phi(a)+\phi(b)$ and $\phi(a\odot b)=\phi(a)\odot \phi(b)$,\\
iii) $a\leq b$ implies that $\phi(a)\leq \phi(b)$, for every $a,b\in L_{1}$.\\
\begin{proposition} Let $T_{1}, T_{2}$ be two isomorphic L-systems. Then $h(T_{1})=h(T_{2})$.
\end{proposition}
\begin{proof} Let $\xi=\{x_{1}, x_{2}, ..., x_{n}\}$ be an arbitrary partition of $L_{1}$. Then $\phi(\xi)=\{\phi(x_{1}), \phi(x_{2}),..., \phi(x_{n})\}$ is a partition of $L_{2}$.\\
$H(\phi(\xi))=-\sum_{i=1}^n m_{2}(\phi(x_{i}))log\hspace{1mm}m_{2}(\phi(x_{i}))=-\sum_{i=1}^n m_{1}(x_{i})log\hspace{1mm}m_{1}(x_{i})=H(\xi)$.\\
$h(T_{2},\phi(\xi))=lim_{n\rightarrow \infty}\frac{1}{n}H(\bigvee_{i=0}^{n-1}T_{2}^i(\phi(\xi)))=lim_{n\rightarrow \infty}\frac{1}{n}H(\bigvee_{i=0}^{n-1}T_{2}^i o\phi(\xi))=lim_{n\rightarrow \infty}\frac{1}{n}H(\bigvee_{i=0}^{n-1}(\phi o T_{1}^i)\xi)=lim_{n\rightarrow \infty}\frac{1}{n}H(\phi(\bigvee_{i=0}^{n-1}T_{1}^i\xi))=lim_{n\rightarrow \infty}\frac{1}{n}H(\bigvee_{i=0}^{n-1}T_{1}^i \xi)=h(T_{1},\xi)$. Then $h(T_{1})=h(T_{2})$.
\end{proof}
\begin{proposition} Let $(L,T,m)$ be an L-system. Then the followings are satisfied:\\
i) $h(T^k)=kh(T)$, for each $k\in N$,\\
ii) if $T:L\rightarrow L$ is invertible, then $h(T^k)=\mid k\mid h(T)$, for each $k\in Z$.
\end{proposition}
\begin{proof} Let $\xi$ be a finite partition of $L$ and $k\in N$.\\
$h(T^k, \xi)=lim_{n\rightarrow \infty}\frac{1}{n}H(\bigvee_{j=0}^{n-1}(T^{kj}(\bigvee_{i=0}^{k-1}T^i \xi)))=lim_{n\rightarrow \infty}\frac{1}{n}H(\bigvee_{j=0}^{n-1}(\bigvee_{i=0}^{k-1}T^{kj+i} \xi))=lim_{n\rightarrow \infty}\frac{1}{n}H(\bigvee_{i=0}^{nk-1}T^i \xi)=lim_{n\rightarrow \infty}\frac{k}{nk}H(\bigvee_{i=0}^{nk-1}T^i \xi)=klim_{n\rightarrow \infty}\frac{1}{nk}H(\bigvee_{i=0}^{nk-1}T^i \xi)=kh(T,\xi)$.\\
$kh(T)=ksup\{h(T,\xi): \xi\hspace{1mm} is\hspace{1mm}  a \hspace{1mm} finite\hspace{1mm}  partition \hspace{1mm} of \hspace{1mm} L\}=sup\{h(T,\bigvee_{i=0}^{k-1}T^i \xi): \xi\hspace{1mm}  is \hspace{1mm} a \hspace{1mm} finite \hspace{1mm} partition\hspace{1mm}  of\hspace{1mm}  L\}=sup\{h(T^k,\bigvee_{i=0}^{k-1}T^i \xi): \xi\hspace{1mm}  is\hspace{1mm}  a \hspace{1mm} finite \hspace{1mm} partition\hspace{1mm}  of\hspace{1mm}  L\}\leq sup\{h(T^k,\xi): \xi \hspace{1mm} is\hspace{1mm}  a\hspace{1mm}  finite\hspace{1mm}  partition\hspace{1mm}  of\hspace{1mm}  L\}=h(T^k)$.\\
$h(T^k,\xi)\leq h(T^k,\bigvee_{i=0}^{k-1}T^i \xi)=kh(T,\xi)$. Then $h(T^k)\leq kh(T)$. Therefore $h(T^k)=kh(T)$, for all $k\in N$.\\
ii) if $k\in Z$ and $k\leq 0$. It is enough to take $-k$ in the part (i) instead of $k$.
\end{proof}
\begin{proposition} Let $(L,T,m)$ be an L-system and $\xi$ a finite partition of $L$. Then $h(T,\xi)=lim_{n\rightarrow \infty}H(\xi \mid \bigvee_{i=1}^nT^i \xi)=H(\xi \mid \bigvee_{i=1}^\infty T^i \xi)$.
\end{proposition}
\begin{proof} Since $\bigvee_{i=0}^{n-1}T^i \xi  \bigvee_{i=0}^n T^i\xi$, $H(\xi \mid \bigvee_{i=0}^{n-1}T^i \xi)\geq H(\xi \mid \bigvee_{i=0}^n T^i \xi)$. Then $lim_{n\rightarrow \infty}H(\xi \mid \bigvee_{i=0}^{n-1}T^i\xi)$ exists. Let $k=1$, then $H(\bigvee_{i=0}^0 T^i \xi)= H(\xi)$. Now suppose that $H(\bigvee_{i=0}^{k-1}T^i \xi)=H(\xi)+\sum_{j=0}^{k-1}H(\xi\mid (\bigvee_{i=1}^j T^i \xi))$. We show that the equality is true for $k+1$.\\
$H(\bigvee_{i=0}^kT^i \xi)=H(\bigvee_{i=1}^kT^i \xi \bigvee \xi)=H(\bigvee_{i=1}^kT^i \xi)+H(\xi \mid \bigvee_{i=1}^kT^i \xi)=H(\bigvee_{i=0}^{k-1}T^{i+1}\xi)+H(\xi \mid \bigvee_{i=1}^kT^i \xi)=H(\bigvee_{i=0}^{k-1}T^i \xi)+H(\xi \mid \bigvee_{i=1}^kT^i \xi)=H(\xi)+\sum_{j=1}^{k-1}H(\xi \mid \bigvee_{i=1}^j T^i \xi)+H(\xi \mid \bigvee_{i=1}^k T^i \xi)=H(\xi)+\sum_{j=1}^kH(\xi \mid \bigvee_{i=1}^jT^i \xi)$. Then for every natural number $n\in N$ it is proved that:
\begin{center}
$H(\bigvee_{i=0}^{n-1}T^i \xi)=H(\xi)+\sum_{j=1}^{n-1}H(\xi \mid \bigvee_{i=1}^jT^i \xi)$.\\
$lim_{n\rightarrow \infty}\frac{1}{n}H(\bigvee_{i=0}^{n-1}T^i \xi)=lim_{n\rightarrow \infty}\frac{1}{n}H(\xi)+lim_{n\rightarrow \infty}\frac{1}{n}\sum_{j=1}^{n-1}H(\xi\mid \bigvee_{i=1}^jT^i \xi)$. Now let $x_{j}=H(\xi \mid \bigvee_{i=1}^jT^i \xi),\hspace{1mm} a_{n}=\sum_{j=1}^{n-1}x_{j}\hspace{1mm}and \hspace{1mm}b_{n}=n$.\\
$\lim_{n\rightarrow \infty}\frac{a_{n+1}-a_{n}}{b_{n+1}-b_{n}}=lim_{n\rightarrow \infty}x_{n+1}=lim_{n\rightarrow \infty}x_{n}=lim_{n\rightarrow \infty}H(\xi \mid \bigvee_{i=1}^n T^i \xi)=H(\xi \mid \bigvee_{i=1}^\infty T^i \xi)$. Then by Cezaro limit Theorem $lim_{n\rightarrow \infty}\frac{1}{n}\sum_{j=1}^{n-1}H(\xi \mid \bigvee_{i=1}^j T^i \xi)=lim_{n\rightarrow \infty}x_{n}=H(\xi \mid \bigvee_{i=1}^\infty T^i \xi)$. Therefore $h(T,\xi)=lim_{n\rightarrow \infty}H(\xi \mid \bigvee_{i=1}^n T^i \xi)=H(\xi \mid \bigvee_{i=1}^\infty T^i \xi)$.
\end{center}
\end{proof}
\begin{corollary} Let $(L,T,m)$ be an L-system and $\xi$ be a finite partition of $L$. Then $h(T,\xi)=0$ if and only if $\xi \stackrel{\circ}\subset \bigvee_{i=1}^\infty T^i \xi$.
\end{corollary}
\begin{corollary} Let $(L,T,m)$ be an L-system. Then $h(T)=0$ if and only if for every finite partition $\xi$ of $L$, $\xi \stackrel{\circ}\subset\bigvee_{i=1}^\infty T^i \xi$. 
\end{corollary}
\begin{definition} Let $(L,T,m)$ be an L-system and $\xi$ be a finite partition of $L$. Then $\xi$ is said  to be generator of the dynamical system $T$ if there exists $n\in N$ such that for each finite partition $\eta$ of $L$ one has $\eta \bigvee_{i=1}^nT^i \xi$.
\end{definition}
\begin{proposition} Let $(L,T,m)$ be an L-system and $\xi$ a generator of $T$. Then $h(T)=h(T,\xi)$.
\end{proposition}
\begin{proof} Let $\eta$ be a finite partition of $L$. Then $\eta\bigvee_{i=1}^n T^i \xi$ and $h(T,\eta)\leq h(T,\bigvee_{i=1}^nT^i \xi)=h(T,\xi)$. Hence $h(T)\leq h(T, \xi)\leq  h(T)$. Therefore $h(T)=h(T,\xi)$.
\end{proof}
\section{Information gain of an L-algebra} \label{info}
The concept of information gain is intimately linked to that of entropy of a random variable, a fundamental  notion in information theory that quantifies the expected amount of information held in a random variable. 
\begin{definition} Let $\xi$ and $\eta$ be two finite partitions of an L-algebra $L$. The information gain of $\xi$ and $\eta$ is defined as:
\begin{center}
$I(\xi,\eta)=H(\xi)-H(\xi \mid \eta)$.
\end{center} 
\end{definition}
\begin{proposition} Let $\xi$ and $\eta$ be two finite partitions of $L$. Then the followings are satisfied:\\
i) $I(\xi,\eta)=H(\xi)+H(\eta)-H(\xi \vee \eta)$,\\
ii) $I(\xi,\eta)=I(\eta, \xi)$,\\
iii) $0\leq I(\xi,\eta)\leq min\{H(\xi),H(\eta)\}$,\\
iv) $\xi \stackrel{\circ} = \eta$ implies that $I(\xi,\zeta)=I(\eta,\zeta)$, for every finite partition $\zeta$  of $L$.
\end{proposition}
\begin{proof} i) $I(\xi,\eta)=H(\xi)-H(\xi \mid \eta)=H(\xi)-(H(\xi \vee \eta)-H(\eta))=H(\xi)+H(\eta)-H(\xi \vee \eta)$,\\
iv) $I(\xi,\zeta)=H(\xi)-H(\xi \vee \zeta)$ and $I(\eta,\zeta)=H(\eta)-H(\eta \vee \zeta)$. Since $\xi \stackrel{\circ} = \eta$, $H(\xi)=H(\eta), H(\xi \mid \zeta)=H(\eta \mid \zeta)$, for each finite partition $\zeta$. Therefore $I(\xi, \zeta)=I(\eta,\zeta)$, for each finite partition $\zeta$ of $L$.
\end{proof}
\begin{example} Let $L=\{0,a,b,1\}$ as in Example 5. Then one can see that $I(\xi,\eta)=0$
\end{example}
\begin{definition} The conditional information gain of $\xi$ and $\eta$ given $\zeta$ is defined as:
\begin{center}
$I(\xi \vee \eta \mid \zeta)=H(\xi \mid \zeta)-H(\xi\mid \eta \vee \zeta)$.
\end{center}
\end{definition}
\begin{proposition} Let $\xi_{1},\xi_{2},...,\xi_{n}$ and $\eta$ be finite partitions of $L$. Then $I(\bigvee_{i=1}^n \xi_{i}, \eta)=I(\xi_{1},\eta)+\sum_{i=2}^n I(\xi_{i},\eta \mid \bigvee_{k=1}^{i-1}\xi_{k})$.
\end{proposition}
\begin{proof} $I(\bigvee_{i=1}^n \xi_{i},\eta)=H(\bigvee_{i=1}^n \xi_{i})-H(\bigvee_{i=1}^n \xi_{i}\mid \eta)=\sum_{i=1}^n H(\xi_{i}\mid \bigvee_{k=1}^{i-1}\xi_{k})-\sum_{i=1}^n H(\xi_{i}\mid \bigvee_{k=0}^{i-1}\xi_{i}\bigvee \eta)=\sum_{i=1}^n H(\xi \mid \bigvee_{k=0}^{i-1}\xi_{k})-H(\xi_{i} \mid \bigvee_{k=0}^{i-1}\xi_{k}\bigvee\eta)=\sum_{i=1}^n I(\xi_{i},\eta \mid \bigvee_{k=0}^{i-1}\xi_{k})$.
\end{proof}
\begin{corollary} If $m: L\rightarrow [0,1]$ be an independent state on $L$, having Bay's property. Then $I(\xi,\eta)=H(\xi)H(\eta)$
\end{corollary}
\begin{definition} Let $\xi$, $\eta$ and $\zeta$ be finite partitions of $L$. $\xi$ is said to be conditionally independent to $\eta$ given $\zeta$, denoted by $\xi\rightarrow(\eta\rightarrow \zeta)$ if $I(\xi,\zeta \mid \eta)=0$.
\end{definition}
\begin{proposition} By the above assumptions:
\begin{center}
$\xi \rightarrow(\eta \rightarrow \zeta)=\zeta\rightarrow(\eta \rightarrow \xi)$.
\end{center}
\end{proposition}
\begin{proof} Let $\xi \rightarrow(\eta \rightarrow \zeta)$, then $I(\xi,\zeta\mid \eta)=0$.\\
$H(\xi \mid \eta)=H(\xi \mid \eta \vee \zeta)=H(\xi \vee (\eta \vee \zeta))-H(\eta \vee \zeta)$. $I(\zeta, \xi \mid \eta)=H(\zeta \mid \eta)-H(\zeta \mid \xi \vee \eta)=H(\xi \vee \eta)-H(\eta)-H((\xi \vee \eta)\vee \zeta)+H(\xi \vee \eta)=H(\xi \vee \eta)-H(\eta)-H(\xi \mid \eta)=0$. Therefore $\zeta \rightarrow (\eta \rightarrow \xi)$. 
\end{proof}
\begin{proposition} Let $\xi,\eta,\zeta$ be finite partitions of $L$. Then
$I(\xi,\eta \vee \zeta)=I(\xi,\eta)+ I(\xi,\zeta\mid \eta)=I(\xi,\zeta)+I(\xi,\eta \mid \zeta)$
\end{proposition}
\begin{proof} $I(\xi,\eta)+I(\xi,\zeta\mid \eta)=H(\xi)-H(\xi \mid \eta)+H(\xi \mid \eta)-H(\xi \mid \eta \vee\zeta)=H(\xi)-H(\xi\mid \eta \vee \zeta)=I(\xi,\eta \vee \zeta)$.
\end{proof}
\begin{corollary} Let $\xi \rightarrow(\eta \rightarrow \zeta)$, then\\
i) $I(\xi \vee \eta,\zeta)=I(\eta,\zeta)$,\\
ii) $I(\eta,\zeta)=I(\xi,\zeta)+I(\zeta,\eta \mid \xi)$,\\
iii) $I(\xi,\eta \mid \zeta)\leq I(\xi, \eta)$.
\end{corollary}
\begin{proof} i) Since $\xi \rightarrow(\eta \rightarrow \zeta)$, $I(\xi,\zeta\mid \eta)=0$. $I(\xi \vee \eta, \zeta)=I(\eta \vee \xi,\zeta)=I(\xi, \zeta)+I(\xi,\zeta \mid \eta)=I(\eta,\zeta)$.\\
ii) By the above proposition and (i)$I(\xi \vee \eta,\zeta)=I(\zeta,\xi)+I(\zeta,\eta \mid \xi)$ and $I(\eta,\zeta)=I(\xi \vee \eta,\zeta)=I(\zeta, \eta)+I(\zeta,\eta \mid \xi)$.\\
iii) $I(\zeta,\eta \mid \xi)\leq I(\zeta,\eta)$, similarly $I(\xi,\eta \mid \zeta)\leq I(\xi,\eta)$.
\end{proof}

\section{Conclusion}\label{coc}
In this article, the concept of an L-algebra was put forward by using the notion of
state on this structure and the entropy of a partition of $L$. The entropy function determines the degree of uncertainty of an event. If we assume that the phenomenon is a set of events, the entropy function is the same as the degree of uncertainty in the occurrence of events. For a set of events, entropy measures the uncertainty (the degree of disorder)of the events in the sample space. Now, if a dynamical system is considered on an L-algebra, the entropy of this system can be defined with the help of what is mentioned in this respect.

.
\end{document}